\pdfoutput=1

\documentclass[a4paper,11pt]{amsart}
\usepackage[plainpages=false]{hyperref}
\usepackage{amsfonts,latexsym,rawfonts,amsmath,amssymb,amsthm,lscape}

\def\TT{{\mathbb{T}}}

\def\NN{{\mathbb{N}}}

\def\ZZ{{\mathbb{Z}}} 

\def\SS{{\mathbb{S}}}

\def\CC{{\mathbb{C}}}
\def\RR{{\mathbb{R}}} 
 
\def\TT{{\mathbb{T}}}

\newtheorem{thm}{Theorem}[section]

\newtheorem{cor}[thm]{Corollary}
\newtheorem{lemma}[thm]{Lemma}

\newtheorem{prop}[thm]{Proposition}
\newtheorem{rmk}[thm]{Remark}
\newtheorem{Def}[thm]{Definition}

\title[Calabi flow on complex tori]{On the extension and smoothing of the Calabi flow on complex tori} 
\author{Hongnian Huang}

\address{
Hongnian Huang\\
Department of Mathematics and Statistics\\
University of New Mexico\\
Albuquerque, NM, 87131\\
U.S.A}

\email{hnhuang@unm.edu}

\date{}

\begin{document}
	
\maketitle

\begin{abstract}
In this paper, we continue to study the Calabi flow on complex tori. We develop a new method to obtain an explicit bound of the curvature of the Calabi flow. As an application, we show that when $n=2$, the Calabi flow starting from a weak K\"ahler metric will become smooth immediately. It implies that in our settings, the weak minimizer of the Mabuchi energy is a smooth one. 
\end{abstract}

\section{Introduction}
In geometrical flows, one of the important tools is the blowup analysis. For example, in Perelman's celebrated work, he uses the blowup analysis to understand the singularities under the Ricci flow in a 3-manifold \cite{P1, P2, P3}. To do this, he needs three tools:

\begin{enumerate}
\item Shi's estimates \cite{Shi} in Ricci flow.
\item Non-collapsing property of the Ricci flow. \cite{P1}
\item Classifications of the $\kappa$-solutions. \cite{P2}
\end{enumerate}

The blowup analysis plays an important role not only in geometrical flows, but also in the continuity methods. The most relevant work to our paper is Donaldson's celebrated proof of the Yau-Tian-Donaldson conjecture in toric surfaces \cite{D1, D2, D3, D4}. In Donaldson's work, he also establishes the Shi-type estimates, the non-collapsing property and the classification of the limiting spaces.

In the Calabi flow, the first work using the blowup analysis is due to Chen-He \cite{ChenHe2, ChenHe3} who study the Calabi flow on Fano toric surfaces. They showed that the Sobolev constant is uniformly bounded along the Calabi flow if the Calabi energy of the initial K\"ahler metric is less than an explicit constant. Thus to show the long time existence of the Calabi flow, they proved

\begin{itemize}
\item Shi-type estimates in integral forms.

\item Classification/non-existence of the limiting spaces.
\end{itemize}

Inspired by the above work, the author proposed a project to study the Calabi flow in toric manifolds \cite{H1}. What we want to establish are the following: 
\begin{enumerate}
\item Shi-type estimates. 
\item Uniform control of Donaldson's $M$-condition: non-collapsing property.
\item Classifications of the limiting spaces.
\end{enumerate}
In \cite{H1}, we classify the limiting spaces. Streets establishes the Shi-type estimates in \cite{St}. Finally,  joint with Feng \cite{FH}, we prove that the Calabi flow exists for all time on $\CC^2 / ( \ZZ^2 + i \ZZ^2)$ with a 2-torus invariant initial metric. 

All the above work require the classification of the limiting space. However, our following theorem suggests that the classification of the limiting spaces may not be necessary to prove the long time existence of the Calabi flow. Let us briefly explain the ideas here. Let $\varphi(t), t \in [0, T)$ be a sequence of relative K\"ahler potentials satisfying the Calabi flow equation. To show that $\varphi(t)$ can be extended over $T$, by Chen-He's compactness theorem (Theorem 1.4 in \cite{ChenHe}), one only needs to show that the Ricci curvature is uniformly controlled in $[0, T)$. The idea of the blowup analysis is that suppose the Ricci curvature is not uniformly bounded in $[0, T)$, then one can pick a sequence of time $t_i \to T^-$, such that
$$
\lim_{i \to \infty} \max_{x \in X} |Rm|(t_i, x) = \infty.
$$
Let $\lambda_i = \max_{x \in X} |Rm|(t_i, x)$. One then proceed to rescale the Calabi flow at $t_i$ by the factor $\lambda_i$. So one obtain a sequence of the Calabi flow $g^{(i)}(t)$. One need to show that as $i \to \infty$, $g^{(i)}(t)$ converges to a limiting Calabi flow $g^{(\infty)}(t)$ in a noncompact K\"ahler manifold $X_\infty$.  This step requires Shi-type estimates and the non-collapsing property. The last step is to obtain a contradiction by studying the behavior of $g^{(\infty)}(t)$ in $X_\infty$. This step usually requires the classification of the limiting spaces. Our observation is that in our settings, there exists an explicit constant $C$ such that if the rescaling factor $\lambda_i \ge C$, then we can obtain a contradiction. Thus we completely avoid the study of the limiting Calabi flow $g^{(\infty)}(t)$ and the limiting K\"ahler manifold $X_\infty$. Our main theorem reads as follows:

\begin{thm}
\label{main}
Let $X = \CC^n / (\ZZ^n + i \ZZ^n).$ Let $\varphi(t) \in \mathcal{H}_\TT, ~ t \in [0, T)$ be a one parameter family of relative K\"ahler potentials satisfying the Calabi flow equation. Suppose that there exists $C_E > 0$ such that for any $t  \in [0, T)$, the total energy 
$$
\left( \int_X |Rm|^n ~ \omega^n(t)  \right)^{\frac{1}{n}}~ \le C_E,
$$ 
then there exists an explicit constant $\lambda > 0$ depending only on $d(\varphi(0), 0), C_E, n$ such that for any $t \in [0, T)$,
$$
|Rm(t)| < \max \left(\lambda, ~ \frac{\lambda}{t^2} \right).
$$
\end{thm}

\begin{rmk}
Our theorem strengthens the results obtained in \cite{FH}:

\begin{itemize}

\item For $n \ge 3$ with the assumption that the total energy is controlled, we are not able to control the Riemann curvature at all in \cite{FH}. 

\item For $n=2$, we only be able to prove an inexplicit bound of the curvature long the Calabi flow in \cite{FH}.

\end{itemize}

We are able to strengthen the results in \cite{FH} because we avoid going to the limiting spaces. In a forthcoming paper, this idea will be generalized to other cases. It seems that the only obstacle to prove the long time existence of the Calabi flow is the non-collapsing property of the Calabi flow. In the setting of our paper, the non-collapsing property only depends on the control of the total energy.
\end{rmk}

Let us discuss an application of our main theorem : the smoothing property of the Calabi flow. Given a weak K\"ahler metric, a natural question is that if we could smooth it and how we smooth it. There are many important work in this area. For example, Chen-Tian-Zhang use the K\"ahler Ricci flow to smooth a weak K\"ahler metric \cite{CTZ}. Similar to \cite{CTZ}, we first need to define a unique weak Calabi flow starting from a weak K\"ahler metric. Then we show that this weak Calabi flow becomes smooth immediately. The existence of the weak Calabi flow follows from Streets' work in \cite{St2} which uses a general theory of Mayer \cite{Ma}. Further development of the weak Calabi flow can be found on \cite{St3, BDL}. The remaining task for us is to show that the weak Calabi flow is a smooth one for $t > 0$.

Let $\mathcal{S}$ be the set of all smooth symplectic potentials and $\mathcal{E}$ be the completion of $\mathcal{S}$ in the sense of the Mabuchi distance. Our result is:

\begin{thm}
\label{smooth}
Let $X = \CC^2 / ( \ZZ^2 + i \ZZ^2) $. Suppose $u_0 \in \mathcal{E}$ and its Mabuchi energy is finite. Let $u(t),~ t \ge 0$ be the weak Calabi flow starting from $u_0$. Then for any $t > 0$, $u(t) \in \mathcal{S}$.
\end{thm}

Since the minimizer of the Mabuchi energy in $\mathcal{E}$ is a fixed point in the weak Calabi flow, we have an immediate corollary:
\begin{cor}
Any minimizer of the Mabuchi energy in $\mathcal{E}$ belongs to $\mathcal{S}$.
\end{cor}

\begin{rmk}
The regularity of weak minimizers of the Mabuchi energy in the general case has been proved by Berman-Darvas-Lu \cite{BDL2}, assuming the existence of smooth cscK metrics. Their result partially confirms conjectures of Chen \cite{Chen} and Darvas-Rubinstein \cite{DR}. It is also shown in \cite{DR} that the regularity of weak minimizers implies the properness of the Mabuchi energy \cite{Ti1, Ti2, DR}.

In our proof, we do not use the fact that there exists cscK metrics in the K\"ahler class. Thus we expect that the Calabi flow can smooth the weak minimizers of the Mabuchi energy in general. As pointed out in \cite{DR}, the regularity problem of the weak minimizers is the main obstacle in proving the existence of cscK metrics, assuming the Mabuchi energy is proper in the K\"ahler class.
\end{rmk}

\section*{Acknowledgement} The author would like to thank Professor Pengfei Guan, Tarm\'as Darvas, Bing Wang, Chengjian Yao for stimulating discussions. 

\section{Introduction}

Let $X = \CC^n / (\ZZ^n + \sqrt{-1} \ZZ^n). $ Let $z_i = \xi_i + \sqrt{-1} t_i, ~ i = 1, \ldots, n$ be the holomorphic coordinates of $\CC^n$. Let $\omega = \frac{\sqrt{-1}}{2} \sum_{i=1}^n z_i \wedge \bar{z}_i$ be the flat K\"ahler metric. Let 
$$
\mathcal{H} = \{\varphi \in C^\infty(X) ~|~ \omega_\varphi = \omega + \sqrt{-1} \partial \bar{\partial} \varphi > 0 \}
$$
be the set of relative K\"ahler potentials. Feng and Szekelyhidi \cite{FS} considered the subset $\mathcal{H}_\TT$ which consists of all the torus invariant relative K\"ahler potentials, i.e.,
$$
\mathcal{H}_\TT = \{ \varphi \in \mathcal{H} ~ | ~ \varphi(z_1, \ldots, z_n) = \varphi(\xi_1, \ldots, \xi_n) \}.
$$
Thus for any $\varphi \in \mathcal{H}_\TT$, $\varphi$ is a smooth, periodic function on $\RR^n$ with period $[-1, 1]^n$ such that $$\psi = \varphi + \frac{1}{2} \sum_{i=1}^n \xi_i^2$$ is a smooth, strictly convex function on $\RR^n$. Feng and Szekelyhidi \cite{FS} considered the Legendre transform of $\psi$: $u$. Let $(x_1, \ldots, x_n) = \nabla \psi(\xi_1, \ldots, \xi_n)$ be the dual coordinates. We have

\begin{lemma}
$f(x) = u(x) - \frac{1}{2} \sum_{i=1}^n x_i^2$ is a periodic function on $\RR^n$ with period $[-1, 1]^n$. 
\end{lemma}

\begin{proof}
We only need to show that $f(x_1, \ldots, x_n) = f(x_1 + 2, x_2, \ldots, x_n)$. Then the dual coordinates of $(x_1+2, x_2, \ldots, x_n)$ is $(\xi_1 + 2, \xi_2, \ldots, \xi_n)$.
Thus
\begin{align*}
& ~f(x_1+2, x_2, \ldots, x_n) - f(x_1, \ldots, x_n) \\
=&~ u(x_1+2, x_2, \ldots, x_n) - u(x_1, \ldots, x_n) - 2 x_1 - 2\\
=& ~(x_1 + 2) (\xi_1 + 2) - x_1 \xi_1 - \psi(\xi_1 + 2, \xi_2, \ldots, \xi_n) + \psi(\xi_1, \ldots, \xi_n) - 2 x_1 - 2\\
=&~0.
\end{align*}

\end{proof}

Let $\mathcal{S}$ be the set of all smooth periodic functions $f(x)$ with period $[-1, 1]^n$ such that $u(x) = f(x) + \frac{1}{2} \sum_{i=1}^n x_i^2$ is a smooth, strictly convex function. Then there exists a one-to-one correspondence between $\mathcal{H}_\TT$ and $\mathcal{S}$ through the Legendre transform. We also denote $\mathcal{E}$ to be the completion of $\mathcal{S}$ in the sense of the Mabuchi distance.

Let $\varphi(t), t \in [0, T)$ be a Calabi flow in $\mathcal{H}_\TT$ be a one parameter family of relative K\"ahler potentials satisfying the Calabi flow equation, i.e.,
$$
\frac{\partial \varphi(t)}{\partial t} = S(t) - \underline{S},
$$
where $S(t)$ is the scalar curvature of $\omega(t) = \omega + \sqrt{-1} \partial \bar{\partial} \varphi(t)$ and $\underline{S} = 0$ is its average. Then the Legendre dual $u(t)$ of $\psi(t) = \varphi(t) + \frac{1}{2} \sum_{i=1}^n \xi_i^2$ satisfies the following equation
$$
\frac{\partial u(t)}{\partial t} = -S(t),
$$
where $S(t) = - \sum_{i j} u^{ij}_{~ij}(t)$ is the Abreu's equation \cite{A1}.

\section{Maximum domain of a special convex function}

In this section, we want to understand if one can have a special (see definition below) convex function $u$ on $\RR^n$. If not, what is the maximum domain in $\RR^n$ one can have for $u$. Let us start with an analytic result:
\begin{prop}
\label{analytic_result}
Let $f ~ : ~ [1, \infty) \to \RR^+$ be a positive function with $$\int_1^\infty \frac{1}{f(x)} ~ dx ~ < ~ M.$$ Then for any $n, C > 0$, there exists $x_0 \in [1, \infty)$ depending only on $C, M, n$ such that $$\int_1^{x_0} f(x) x^{n-1} ~ dx > C x_0^{n+1}.$$ 
\end{prop}

\begin{proof}
	Suppose the conclusion is not true. Then there exists $C, n > 0$ such that for any $x \in [1, \infty)$, we have $$\int_1^x f(x) x^{n-1} ~ dx \le C x^{n+1}.$$ Let us construct a sequence of positive constant $R_i, ~ i = 0, 1, \ldots $ such that $R_i = 2^i$. By H\"older's inequality, we know that
\begin{align*}
&\int_{R_i}^{R_{i+1}} \frac{1}{f(x)}~ dx \\
\ge ~& \frac{\left(\int_{R_i}^{R_{i+1}} x^\frac{n-1}{2} ~ dx\right)^2}{\int_{R_i}^{R_{i+1}} f(x) x^{n-1} ~ dx} \\
\ge~ & \frac{\left(\frac{2}{n+1} \left(R_{i+1}^\frac{n+1}{2} - R_i^\frac{n+1}{2} \right) \right)^2}{2^{n+1} C R_i^{n+1}} \\
 = ~ & \frac{4 \left( 2^\frac{n+1}{2} - 1\right)^2}{(n+1)^2 2^{n+1} C}.
\end{align*}
	
Let $i \to \infty$, we obtain a contradiction. In fact, for any given $n, C > 0$, there exists $x_0$ with
$$
0 \le \ln x_0 \le \left( [\frac{M (n+1)^2 2^{n+1} C}{4 \left( 2^\frac{n+1}{2} - 1\right)^2}] + 1\right) \ln 2
$$such that $\int_1^{x_0} f(x) x^{n-1} ~ dx > C x_0^{n+1}$, where $[x]$ denotes the integer part of $x$.
\end{proof}

Now let us define what a special convex function of type $(M, C_0, C_E)$ is.
\begin{Def}
\label{standard}
$u$ is a special convex function on $\RR^n$ of type $(M, C_0, C_E)$ if $u$ satisfies the following condition:
\begin{itemize}
	\item $u$ is a smooth, strictly convex function. 
	\item $u(0, \ldots, 0) = 0, ~ D u(0, \ldots, 0) = (0, \ldots, 0)$.
	\item For any $x \in \RR^n$, $|Du(x)| < M$.
	\item For any $x$ outside the Euclidean Ball $B_E(O, 1)$, $u_r(x) > C_0$.
	\item $$ \left( \int_{\RR^n} \left( |u^{ij}_{~ij} |\right)^n ~ d \mu \right)^{\frac{1}{n}} < C_E.$$
\end{itemize}
\end{Def}
Let $R > 0$ be a positive constant and let 
$$
f(x) =
\left\{
\begin{array}{ll}
(u(x) - R)^2, & \mathrm{if} ~ u(x) < R;\\
0, & \mathrm{otherwise}.
\end{array}
\right.
$$

We have
\begin{align*}
0 \le &\int_{\RR^n} u^{ij} f_i f_j ~ d \mu \\
=& \frac{1}{2} \int_{\mathrm{supp}(f)} u^{ij} (f^2)_{ij} - 2 u^{ij} f_{ij} f ~ d \mu\\
=& \frac{1}{2} \int_{\mathrm{supp}(f)} u^{ij}_{~ij} (f^2) - 4 n (u - R)^3 - 4 u^{ij} u_i u_j f ~ d \mu\\
\le & \frac{1}{2} \| S(u) \|_{\frac{1}{n}}  \| f^2 \|_{\frac{n-1}{n}} + 2 \int_{\mathrm{supp}(f)} -n (u - R)^3 - u^{ij} u_i u_j f ~ d \mu.\\
\end{align*}
It yields
\begin{equation}
\label{inequality}
- \int_{\mathrm{supp}(f)} n (u - R)^3 ~ d \mu + \frac{C_E}{4} \| f^2 \|_{\frac{n-1}{n}} \ge \int_{\mathrm{supp}(f)} u^{ij} u_i u_j f ~ d \mu.
\end{equation}

Let us analyze the left hand side of the inequality (\ref{inequality}). Since $u$ is a special convex function, for any $x \in \RR^n \backslash B_E(O, 1)$, we have $u_r > C_0$ where $r = d_E(O, x)$. Combining with the fact that $|D u| < M$, we have

\begin{lemma}
There exists a constant $C_1 > 0$ depending only on $n, C_0, M$ such that 
$$
\frac{R^n}{C_1} < Vol(\mathrm{supp}(f)) < C_1 R^n.
$$
\end{lemma}

\begin{proof}
The proof of this lemma is elementary and we leave it to the interested readers.
\end{proof}

As a corollary, we have

\begin{cor}
$$
- \int_{\mathrm{supp}(f)} n (u - R)^3 ~ d \mu + \frac{C_E}{4} \| f^2 \|_{\frac{n-1}{n}} \le C_2 R^{n+3},
$$
where $C_2$ depends only on $n, C_0, M, C_E$.
\end{cor}

\begin{proof}
It is easy to see that 
\begin{align*}
&- \int_{\mathrm{supp}(f)} n (u - R)^3 ~ d \mu \\
\le & R^3 Vol(\mathrm{supp}(f)) \\
\le & C R^{n+3}.
\end{align*}

For the second term, we have

\begin{align*}
& \| f^2 \|_{\frac{n-1}{n}} \\
\le & \left(  \int_{\mathrm{supp}(f)} f^\frac{2n}{n-1} ~ d \mu \right)^{\frac{n-1}{n}}\\
\le & R^4 \left(Vol(\mathrm{supp}(f)) \right)^\frac{n-1}{n}\\
\le & C R^{n+3}.
\end{align*}

Combining the above two terms, we obtain the conclusion.
\end{proof}

Now let us analyze the right hand side of the inequality (\ref{inequality}). Let $$(r, \theta = (\theta_1, \ldots, \theta_{n-1}))$$ be the spherical coordinates. Since $0 < D_r u (r, \theta) < M$ for any $r > 0, ~ \theta \in \SS^{n-1}$, for any $0 < r \le \frac{R}{2 M}$, we have $u(r, \theta) < R/2$. Using the coordinate system as in Lemma 3 of \cite{D2} and Lemma 4.3 of \cite{H1} (without changing the value of $u_{rr}$), direct calculations show that:
\begin{align*}
& \int_{\mathrm{supp}(f)} u^{ij} u_i u_j f ~ d \mu \\
\ge & \int_{\SS^{n-1}} \int_1^{\frac{R}{2 M}} u^{rr} u_r u_r \frac{R^2}{4} ~ r^{n-1} d r d \sigma_{\SS^{n-1}} \\
\ge & \frac{R^2 C_0^2}{4} \int_{\SS^{n-1}} \int_1^{\frac{R}{2 M}} u^{rr} ~ r^{n-1} d r d \sigma_{\SS^{n-1}} \\
= & \frac{R^2 C_0^2}{4} \int_{\SS^{n-1}} \int_1^{\frac{R}{2 M}} \frac{1}{u_{rr}} ~ r^{n-1} d r d \sigma_{\SS^{n-1}}.
\end{align*}

Let $C_3 > 0$ be a constant to be determined later. Then by Proposition (\ref{analytic_result}), there exists $R_0 > 1$ such that 
$$
\int_1^{\frac{R_0}{2 M}} u^{rr} ~ r^{n-1} d r > C_3  \left( \frac{R_0}{2M} \right)^{n+1}.
$$

Thus we have
\begin{align*}
C_2 R_0^{n+3} \ge C_3 \frac{R_0^2 C_0^2}{4} Vol(\SS^{n-1}) \left(\frac{R_0}{2M}\right)^{n+1}.
\end{align*}

We can choose $C_3 = \frac{4 C_2 (2M)^{n+1}}{ C_0^2 Vol(\SS^{n-1})} + 1$ to obtain a contradiction. By Proposition (\ref{analytic_result}), we have the following theorem:

\begin{thm}
\label{max_domain}
For a special function of type $(M, C_0, C_E)$. Its domain is contained in the Euclidean ball centered at $O$ with radius
$$
1 + \frac{R_0}{C_0},
$$
where 
$$
R_0 = 2M \exp \left(  \left( [\frac{M (n+1)^2 2^{n+1} C_3}{4 \left( 2^\frac{n+1}{2} - 1\right)^2}] + 1\right) \ln 2 \right).
$$
\end{thm}

\section{Modified Blowup analysis}
We will prove Theorem (\ref{main}) in this section.  In \cite{FH}, we obtained the following results:

\begin{prop}
There exists $M > 0$ depending only on $n, ~ \int_{[-1, 1]^n} u^2(0) ~ d \mu$ such that for any $t \in [0, T), ~ x \in [-2, 2]^n$, we have
$$
|D u(t)| < M,
$$
where $D u(t)$ denotes the Euclidean derivative of $u(t)$.
\end{prop}

\begin{proof}[Proof of Theorem (\ref{main})]

Let $t_0 > 0$ be the first time that there exists $x_0 \in [-1, 1]^n$ such that $|Rm(t_0, x)| \ge \max(\lambda, \frac{\lambda}{\sqrt{2 t_0}})$. The first case we want to consider is $t_0 \ge 1$. Let us rescale the flow by $\lambda$, i.e., we define a new flow
$$
\tilde{u}(t, x) = \lambda u \left( \frac{t - t_0}{\lambda^2}, \frac{x - x_0}{\lambda} \right).
$$
Then following the analysis in Section 4 of \cite{FH}, one obtain that there exists $C_0 > 0$ depending only on $n, M, C_E$ such that for any $x$ outside $B_E(O, 1)$, 
$$
D_r \tilde{u} (0, x) > C_0.
$$
Let $v(x) = \tilde{u}(0, x)$. Thus $v(x)$ is a special convex function of type $(M, C_0, C_E)$ defined on a domain $[-\lambda, \lambda]^n$. Theorem (\ref{max_domain}) gives a contradiction if we choose $\lambda = 2 + \frac{R_0}{C_0}$.

The second case is $t_0 < 1$. Similar to the above case, let us rescale the flow by $\frac{\lambda}{\sqrt{2 t_0}}$. Without loss of generality, let us assume $\lambda > 1$, so we obtain a new flow defined in $[-2, 0]$. If for any $t \in [-1, 0]$, $\max |Rm|(t) < 2$, then we can apply the argument in the first case to obtain a contradiction. If there exists some $t \in [-1, 0]$ such that $\max |Rm|(t) \ge 2$, then the usual point-picking techniques allow us to obtain a flow in $[-1, 0].$ Applying the arguments in the first case, we can also obtain a contradiction.

\end{proof}

\section{Smoothing property}
Let $\mathcal{E}$ be the completion of $\mathcal{S}$ in the sense of the Mabuchi distance. In fact $\mathcal{E}$ is the $L^2$-completion of $\mathcal{S}$. For any $f \in \mathcal{E}$, it is a periodic function on $\RR^n$ with period $P = [-1, 1]^n$. Moreover, $u(x) = f(x) + \frac{1}{2} \sum_{i=1}^n x_i^2$ is a convex function. Thus $(D^2 u)$ exists almost everywhere on $\RR^n$ which implies that $(D^2 f)$ exists almost everywhere on $\RR^n$. Following the ideas of \cite{TW, ZZ}, we consider the mollification $f_h$ of $f$, i.e.,
$$
f_h (x) = h^{-n} \int_{\RR^n} f(y) \eta \left( \frac{x-y}{h}\right) ~ dy
$$
for some nonnegative function $\eta$ supported on the unit ball $B_1(0)$ and satisfying $\int_{\RR^n} \eta = 1$. It is easy to see that $f_h(x) + \frac{1}{2} \sum_{i=1}^n x_i^2$ is a smooth, strictly convex function on $\RR^n$.  For any point $x \in \RR^n$ that $(D^2 f)(x)$ exists, we have $(D^2 f_h)(x) \to (D^2 f)(x)$ as $h \to 0$ \cite{Zi}. Let us consider a sequence of smooth, strictly convex function on $\RR^n$:
$$
u_m(x) = \frac{m}{m+1} f_{\frac{1}{r(m)}} (x) + \frac{1}{2} \sum_{i=1}^n x_i^2,
$$
where $r$ is some function from $\NN$ to $\NN$ such that $r(i) \ge i$ for all $i \in \NN$.

The following lemma is crucial for us:

\begin{lemma} 
\label{approximate}
For any $r : \NN \to \NN$ with $r(i) \ge i$ for all $i \in \NN$, we have
$$
\lim_{m \to \infty} \| u_m(x) - u(x) \|_{L^2(P)} \to 0.
$$
Moreover, if the extended Mabuchi energy at $f$ is finite, i.e., $$\int_P \log \det (D^2 u) ~ dx > - \infty,$$ then 
there exists a $r : \NN \to \NN$ with $r(i) \ge i$ for all $i$ such that
$$
\lim_{m \to \infty} \int_P \log \det (D^2 u_m) ~ dx = \int_P \log \det (D^2 u) ~ dx.
$$
\end{lemma}

\begin{proof}
The first statement is easy to prove:

\begin{align*}
& \lim_{m \to \infty}  \| u_m(x) - u(x) \|_{L^2(P)} \\
=&  \lim_{m \to \infty}  \|\frac{m}{m+1} f_{\frac{1}{r(m)}} (x) - f(x) \|_{L^2(P)}\\
\le & \lim_{m \to \infty}  \frac{m}{m+1} \| f_{\frac{1}{r(m)}} (x) - f(x) \|_{L^2(P)} +  \lim_{m \to \infty}  \|\frac{m}{m+1} f(x) - f(x) \|_{L^2(P)}\\
=& 0.
\end{align*}
For the second statement, let $v_m(x) = f(x) + \frac{m+1}{2m} \sum_{i=1}^n x_i^2$. Since $v_m(x)$ is uniformly bounded on $[-2, 2]^n$, there exists a constant $C > 0$ independent of $m$ such that
$$
\int_P \log \det (D^2 v_m(x)) \le C, \quad \int_P \log \det (D^2 u(x)) \le C.
$$
Also $\det (D^2 v_m(x)) \ge \det (D^2 u(x))$ implies that
$$
\int_P \log \det (D^2 v_m(x)) \ge \int_P \log \det (D^2 u(x)).
$$
Then dominated convergence theorem shows that
$$
\lim_{m \to \infty} \int_P \log \det (D^2 v_m) ~ dx = \int_P \log \det (D^2 u) ~ dx.
$$

Now for every $j \in \NN$, we define 
$$
v_{j, m} (x) := f_{\frac{1}{j}}(x) + \frac{m+1}{2m} \sum_{i=1}^n x_i^2.
$$
Again, $v_{j, m} (x)$ is a smooth, strictly convex function on $\RR^n$. In fact, 
$$
\log \det (D^2 v_{j, m} (x)) \ge - n \log m, \quad \log \det (D^2 v_m (x)) \ge - n \log m.
$$
Following the proof of Lemma 2.2 in \cite{ZZ}, one concludes that
$$
\lim_{j \to \infty} \log \det (D^2 v_{j, m} (x)) = \log \det (D^2 v_m(x)).
$$
Thus for each $m$, one can choose $j = r(m) \ge m $ such that $$| \log \det (D^2 v_{j, m} (x)) - \log \det (D^2 v_m (x)) | < \frac{1}{m}.$$

It is easy to see that with our choice of $r : \NN \to \NN$, we have

$$
\lim_{m \to \infty} \int_P \log \det (D^2 u_m) ~ dx = \int_P \log \det (D^2 u) ~ dx,
$$
where $u_m = \frac{m}{m+1} v_{r(m), m}$.
\end{proof}

Now we are ready to prove Theorem (\ref{smooth}): 

\begin{proof} [Proof of Theorem (\ref{smooth})]
Let $u_0 \in \mathcal{E}$ with finite Mabuchi energy. Then by Lemma (\ref{approximate}), there exists a sequence of $u_m \in \mathcal{S}$ such that
\begin{align*}
& \lim_{m \to \infty} \| u_m(x) - u(x) \|_{L^2(P)} =0, \\
& \lim_{m \to \infty} \int_P \log \det (D^2 u_m) ~ dx = \int_P \log \det(D^2 u) ~ dx. 
\end{align*}

For each $m \in \NN$, let $u_m(t)$ be the Calabi flow starting from $u_m$. Let $v(x) = \frac{1}{2} \sum_{i=1}^2 x_i^2$. It is easy to see that there exists $C_1, C_2 > 0$ such that for any $m \in \NN,~ t  > 0$,
\begin{align*}
\| u_m(t, x) - v(x) \|_{L^2(P)} \le C_1, \quad Ca(u_m(t)) \le \frac{C_2}{t}.
\end{align*}

Then for any fixed $t > 0$, Theorem (\ref{main}) provides uniform control of the curvature of $u_m(t, x)$. Thus for any fixed $t > 0, k \in \NN$, one has uniform control of $C^k$ norm of $u_m(t, x)$. Let $u(t)$ be the weak Calabi flow starting from $u_0$. Then by the fact that the (weak) Calabi flow decreases the Mabuchi distance \cite{CC, St2, Ma}, one has
$$
d(u_m(t), u(t)) \le d(u_m, u_0).
$$

Hence 

$$
\lim_{m \to \infty} u_m(t) = u(t).
$$

We conclude that $u(t)$ is a smooth function.
\end{proof}

\begin{rmk}
It is not hard to construct a K\"ahler metric in $\mathcal{E} \backslash \mathcal{S}$. In fact, let 
$$
f(x, y) = \frac{1}{4} (x^4 + y^4) - \frac{1}{2} (x^2 +y^2). 
$$
Then $f(x, y)$ is a periodic function on $\RR^2$ (up to the second order) with period $[-1, 1]^2$. It is easy to see that $u(x, y) = f(x, y) + \frac{1}{2} (x^2 +y^2)$ in $\mathcal{E} \backslash \mathcal{S}$ and its Mabuchi energy is:

\begin{align*}
& - \int_P \log \det (D^2 u) ~ dx dy \\
=& -\int_{-1}^1 \int_{-1}^1 2 \ln 3 + 2 \ln |x| + 2 \ln |y| ~ dx dy\\
=& - 8 \ln 3 - 16 \int_0^1 \ln x ~ dx\\
=& - 8 \ln 3 + 16.
\end{align*}
\end{rmk}

\end{document}